\documentclass[a4paper,11pt]{amsart}

\usepackage{amssymb}
\usepackage{amsbsy}
\usepackage{amscd}
\usepackage{amsmath}
\usepackage{amsthm}
\usepackage[english]{babel}
\usepackage{cite}
\usepackage{latexsym}
\usepackage[all]{xy}
\usepackage{verbatim}
\usepackage{hyperref}

\newtheorem{inizio}{Lemma}[section]
\newtheorem{theorem}[inizio]{Theorem}
\newtheorem{corollary}[inizio]{Corollary}
\newtheorem{proposition}[inizio]{Proposition}

\newtheorem{definition}[inizio]{Definition}
\newtheorem{remark}[inizio]{Remark}

\newcommand{\QQ}{\mathbb{Q}}
\newcommand{\qq}{\mathbb{Q}}
\newcommand{\ZZ}{\mathbb{Z}}

\newcommand{\pp}{\mathbb{P}}
\newcommand{\PP}{\mathbb{P}}

\newcommand{\oo}{\mathcal{O}}

\newcommand{\mL}{\mathcal{L}}

\newcommand{\EE}{\mathcal{E}}
\newcommand{\lr}{\longrightarrow}
\newcommand{\mP}{\mathbb{P}}

\begin{document}

\title[Deformations of product-quotient surfaces]
{Deformations of product-quotient surfaces and reconstruction of Todorov surfaces
via $\mathbb{Q}$-Gorenstein smoothing}

\author{Yongnam Lee}

\address{Department of Mathematics, Sogang University, Sinsu-dong, Mapo-gu,
Seoul 121-742, Korea}

\email{ynlee@sogang.ac.kr}

\author{Francesco Polizzi}

\address{Dipartimento di Matematica, Universit\`{a} della
Calabria, Cubo 30B, 87036 \\ Arcavacata di Rende (Cosenza), Italy}

\email{polizzi@mat.unical.it}

\subjclass[2010]{Primary 14J29; Secondary 14J10, 14J17}

\keywords{Surfaces of general type,
$\mathbb{Q}$-Gorenstein smoothing, product-quotient surfaces}

\date{}

\setcounter{section}{-1}

\begin{abstract}
We consider the deformation spaces of some singular
product-quotient surfaces $X=(C_1 \times C_2)/G$, where the curves
$C_i$ have genus $3$ and the group $G$ is isomorphic to 
$\mathbb{Z}_4$. As a by-product, we give a new construction of
Todorov surfaces with $p_g=1$, $q=0$ and $2\le K^2\le 8$ by using
$\QQ$-Gorenstein smoothings.
\end{abstract}

\maketitle

\section{Introduction} \label{sec.intro}

In \cite{To81}, Todorov constructed some surfaces of general type
with $p_g=1$, $q=0$ and $2\le K^2\le 8$ in order to give
counterexamples of the global Torelli theorem. Todorov surfaces
with $K^2=8-k$ are double covers of a Kummer surface in
$\mathbb{P}^3$ branched over a curve $D$, which is a complete
intersection of the Kummer surface with a smooth quadric surface
containing $k$ of its nodes, and over the remaining $16-k$ nodes.
Surfaces with $K^2=2$, and $p_g=1$ have been completely classified by
Catanese and Debarre \cite{CD89}, while some examples were constructed by Todorov.
C. Rito \cite{Rito09} gave a detailed study of Todorov surfaces with an involution.

Recently, H. Park, J. Park and D. Shin constructed simply
connected surfaces of general type with $p_g=1$, $q=0$ and $2\le
K^2\le 8$ by considering $\QQ$-Gorenstein smoothings of singular
K3 surfaces with special configurations of cyclic quotient
singularities, see \cite{PPS1}, \cite{PPS2}. Their construction
follows the method used by Lee and Park in the paper \cite{LP07},
where a simply connected surface of general type with $p_g=q=0$
and $K^2=2$ is constructed via the $\QQ$-Gorenstein smoothing of a
singular rational surface. For more details about these kind of
techniques, over a field of any characteristic, we refer the
reader to the work of Lee and Nakayama \cite{LN11}.

Moreover, Bauer, Catanese, Grunewald and Pignatelli constructed
many interesting examples of surfaces of general type with $p_g=0$
by considering the minimal desingularization of singular
product-quotient surfaces, see  \cite{BC04}, \cite{BCG08},
\cite{BCGP}, \cite{BP}. Similar methods are applied to surfaces of
general type with $p_g=q=1$  by Polizzi and others, see
\cite{Pol08}, \cite{Pol09}, \cite{CP09}, \cite{MP10}. These
results motivated us to start the investigation of
$\QQ$-Gorenstein smoothings of singular product-quotient surfaces.

Let us recall that a projective surface $S$ is called a
\emph{product-quotient surface} if there exists a finite group
$G$, acting faithfully on two smooth curves $C_1$ and $C_2$ and
diagonally on their product, so that $S$ is isomorphic to the
minimal desingularization of $X = (C_1 \times C_2)/G$. The surface
$X$ is called a \emph{singular model of a product-quotient
surface}, or simply a \emph{singular product-quotient surface}.

This paper focuses on the case $g(C_1)=g(C_2)=3$ and
$G=\mathbb{Z}_4$. More precisely, we assume that there exist two
simple $\ZZ_4$-covers $g_i \colon C_i \to \mathbb{P}^1$, both
branched in four points. Then the singular product-quotient
surface
\begin{equation*}
X:=(C_1\times C_2)/\mathbb{Z}_4
\end{equation*}
contains precisely $16$ cyclic quotient singularities; any of them
is either of type $\frac{1}{4}(1, \,1)$ or of type $\frac{1}{4}(1,
\, 3)$. Note that $\frac{1}{4}(1, \, 3)$ is a rational double
point, whereas $\frac{1}{4}(1, \,1)$ is a singularity of class
$T$, so both admit a \emph{local} $\mathbb{Q}$-Gorenstein
smoothing, see \cite{KSB88} or \cite[Sections 2-4]{Man08}. The problem is to
understand whether these local smoothings can be glued together in
order to have a \emph{global} $\mathbb{Q}$-Gorenstein smoothing of
$X$. We will show that in some cases this is actually possible.

This paper is organized as follows.

In Section \ref{sec.prel} we present some preliminaries and we set
up notation and terminology. In particular, we recall the
definitions of simple cyclic cover of a curve and of singular
product-quotient surface and we explain how to compute their basic
invariants.

In Section \ref{sec.basic.construction} we introduce the main
objects that we want to study, namely the singular product
quotient surfaces of the form  $X=(C_1 \times C_2)/G$, where
$g(C_1)=g(C_2)=3$, $G=\mathbb{Z}_4$ and $C_i \to C_i/G$ is a
simple cyclic cover for $i=1, \, 2$.

Section \ref{sec.example.Y} deals with the study of the singular
product-quotient surface $Y=(C_1 \times C_2)/H$, where $H$ is the
unique subgroup of $G$ isomorphic to $\mathbb{Z}_2$. By
construction,  $Y$ contains exactly $16$ ordinary double points as
singularities. By using the infinitesimal techniques introduced in
\cite{Pin81} and \cite{Cat89}, we prove that $\textrm{Def}(Y)$ is
smooth at $Y$, of dimension $18$ and $\textrm{ESDef}(Y)$ is smooth
at $[Y]$, of dimension $8$ (Proposition \ref{prop.def(Y)}).
Moreover, if $\mu \colon V \to Y$ is the minimal desingularization
of $Y$, we have
\begin{equation*}
\dim_{[V]} \textrm{Def}(V)=18, \quad h^1(\Theta_V)=24,
\end{equation*}
hence $\textrm{Def}(V)$ is singular at $[V]$; by \cite{BW74}
 this implies that the sixteen $(-2)$ curves of $V$ do not have
 independent behavior in deformations.

In Section \ref{sec.example.X} we discuss three examples of
singular product-quotient surface $X=(C_1 \times C_2)/G$ with
different $G$-action.
\begin{itemize}
\item In the first example we have $\textrm{Sing}(X)=16 \times
\frac{1}{4}(1, \, 3)$, so $X$ contains only rational double points
as singularities. We prove that $\textrm{Def}(X)$ and
$\textrm{ESDef(X)}$ are both smooth at $[X]$, of dimension $44$
and $2$, respectively (Propositions \ref{dim44} and
\ref{prop.example.1}).

The surface $X$ satisfies $h^0(\omega_X)=5$ and $K_X^2=8$; moreover
it is no difficult to see that the canonical map $\phi_K \colon X
\to \mathbb{P}^4$ is a birational morphism onto its image; by
\cite[Proposition 6.2]{Cat97} it follows that the general
deformation of $X$ is isomorphic to a smooth complete intersection
of bidegree $(2, \,4)$ in $\mathbb{P}^4$.

Moreover we have
\begin{equation*}
\dim_{[S]} \textrm{Def}(S)=44, \quad h^1(\Theta_S)=50,
\end{equation*}
hence $\textrm{Def}(S)$ is singular at $S$. This means that the
sixteen $A_3$-cycles of $S$ do not have independent behavior in
deformations.

\item In the second example we have  $\textrm{Sing}(X)=16 \times
\frac{1}{4}(1, \, 1)$. We show that there exist a
$\mathbb{Q}$-Gorenstein smoothing $\pi \colon \mathcal{X} \to T$
of $X$, whose base $T$ has dimension $12$, such that the general
fibre $X_t$ of $\pi$ is a minimal surface of general type whose invariants
are
\begin{equation*}
p_g(X_t)=1, \quad q(X_t)=0, \quad K_{X_t}^2=8.
\end{equation*}
Moreover $X_t$ is isomorphic to a Todorov surface with $K^2=8$
(Theorem \ref{teo.example.2}). By a slight modification of the
construction, it is possible to obtain all Todorov surfaces with
$2 \leq K^2 \leq 8$.

This is related to the existence of complex structures on rational
blow-downs of algebraic surfaces. More precisely, one can consider
the rational blow-down $S(t)$ of $t$ of the $(-4)$-curves in $S$,
where $1 \leq t \leq 16$. This means that one considers the normal
connected sum of $S$ with $t$ copies of $\mathbb{P}^2$,
identifying a conic in each $\mathbb{P}^2$ with a $(-4)$-curve in
$S$; then $S(t)$ is a symplectic $4$-manifold. On
can therefore raise the following:

\noindent{\bf Question.} Is it possible to give a complex structure on
$S(t)$ for $1\le t \le 16$, and to describe $S(t)$ when such a
complex structure exists?

Our results  answer affirmatively this question when $10\le t\le
16$; in these cases, indeed, one can give a complex structure to
the rational blow-down $S(t)$, which make it isomorphic to a
Todorov surface with $K^2=t-8$.

\item In the third example, we have $\textrm{Sing}(X)=8 \times
\frac{1}{4}(1, \, 1)+ 8 \times \frac{1}{4}(1, \, 3)$. Rasdeaconu
and Suvaina give an explicit construction of the minimal
desingularization $S$ of $X$, see \cite[Section 3]{RS06}; in fact,
they prove that $S$ is a simply connected, minimal elliptic
surface with no multiple fibres.

We show that there exists a $\mathbb{Q}$-Gorenstein smoothing of
$X$, although $H^2(\Theta_X)\ne 0$ and all the natural deformations
of the $G$-cover $u \colon X \to Q$ preserve the $8$ singularities
of type $\frac{1}{4}(1, \,1)$, see Proposition \ref{pro.example.3}.
Indeed we prove that a general surface $\bar X$ in the subfamily of
natural deformations of the $G$-cover of $X$ can be deformed to a
bidouble cover of $\PP^1\times\PP^1$ branched over three smooth
divisors of bidegree $(2, \, 2)$. By taking a general deformation of
these three divisors we obtain a $\QQ$-Gorenstein smoothing of $X$
which smoothes all the singularities. More generally, by using the
same method one can construct surfaces of general type with $p_g=3$,
$q=0$ and $K^2=k$ $(2\le k\le 8)$ by first taking a $\QQ$-Gorenstein
smoothing of $k$ singular points of type $\frac{1}{4}(1, \, 1)$ of
$\bar X$ and then the minimal resolution of the remaining $8-k$
singular points of the same type.
\end{itemize}

\noindent\textbf{Acknowledgments.} Both authors were partially supported by
the World Class University program through the National Research
Foundation of Korea funded by the Ministry of Education, Science
and Technology (R33-2008-000-10101-0). Both authors
appreciate M. Reid for valuable suggestions.

Yongnam Lee thanks KIAS for the invitation as an affiliate member;
part of this paper was worked out during his visit to KIAS.

Francesco Polizzi was partially supported by  the Progetto MIUR di
Rilevante Interesse Nazionale \emph{Geometria delle
Variet$\grave{a}$ Algebriche e loro Spazi di Moduli}. He thanks
the Department of Mathematics of Sogang University for the
invitation in the winter semester of the academic year 2009-2010
and the Mathematisches Institut-Universit\"at Bayreuth for the
invitation in the period October-November 2011. He is also
grateful to I. Bauer and F. Catanese for stimulating discussions
and useful suggestions.

\medskip

\textbf{Notation and conventions.}

We work over the field $\mathbb{C}$ of complex numbers.

By ``surface'' we mean a projective, non-singular surface $S$, and
for such a surface $\omega_S=\oo_S(K_S)$ denotes the canonical
class, $p_g(S)=h^0(S, \, \omega_S)$ is the \emph{geometric genus},
$q(S)=h^1(S, \, \omega_S)$ is the \emph{irregularity} and
$\chi(\mathcal{O}_S)=1-q(S)+p_g(S)$ is the \emph{Euler-Poincar\'e
characteristic}.

If $X$ is any (possibly singular) projective scheme, we denote by
$\textrm{Def}(X)$ the base of the Kuranishi family of deformations
of $X$ and by $\textrm{ESDef}(X)$ the base of the equisingular
deformations of $X$. The tangent spaces to $\textrm{Def}(X)$ and
$\textrm{ESDef}(X)$ at the point $[X]$ corresponding to $X$ are
given by $\textrm{Ext}^1(\Omega^1_Y, \, \mathcal{O}_Y)$ and
$H^1(\Theta_Y)$, respectively.

If $L$ is a line bundle $L$ on $X$, we use the notation $L^n$
instead of $L^{\otimes n}$ if no confusion can arise.

If $G$ is any finite abelian group, we denote by $\widehat{G}$ its
dual group, namely the group of irreducible characters of $G$.

\section{Preliminaries} \label{sec.prel}

\subsection{Simple cyclic covers of curves}\label{cyclic.cover}
Let $\Gamma$ be a smooth, projective curve and $B \subset \Gamma$ an
effective divisor such that $\oo_{\Gamma}(B)=\mathcal{L}^n$ for some
$\mathcal{L} \in \textrm{Pic}(\Gamma)$. Therefore there exists a
$\mathbb{Z}_n$-cover $g \colon C \to \Gamma$, totally branched over
$B$, which is called a \emph{simple cyclic cover}. We identify
$\mathbb{Z}_n$ with the group of $n$-th roots of unity, namely
$\mathbb{Z}_n = \langle \zeta \rangle$, where $\zeta$ is a primitive
$n$-th root. The dual group $\widehat{\mathbb{Z}}_n$ is isomorphic
to $\mathbb{Z}_n$, and it is generated by the character $\chi_1
\colon \mathbb{Z}_n \to \mathbb{C}$ such that
$\chi_1(\zeta)=\zeta^{-1}$. We will write $\chi_j$ instead of
$\chi_1^j$; then $\chi_j(\zeta)=\zeta^{-j}$. The group
$\mathbb{Z}_n$ acts naturally on $g_*\oo_C$, so there is a canonical
splitting
\begin{equation}\label{eq.decomp.O}
g_* \oo_C = \oo_{\Gamma} \oplus \mL^{-1} \oplus \ldots \oplus
\mL^{-(n-1)},
\end{equation}
where the summand $\mL^{-j}$ is the eigensheaf $(g_*
\oo_C)^{\chi_j}$ corresponding to the character $\chi_j$.

Similarly, $\mathbb{Z}_n$ acts naturally on $g_*\omega_C$ and
$g_*\omega_C^2$, giving the following decompositions (see
\cite{Pa91} and \cite[Section 2]{Cat89}):
\begin{equation}\label{eq.decomp.omega}
\begin{split}
g_* \omega_C & = \omega_{\Gamma} \oplus (\omega_{\Gamma} \otimes
\mL)
\oplus \ldots \oplus (\omega_{\Gamma}\otimes \mL^{n-1}), \\
g_* \omega_C^2 & = (\omega_{\Gamma}^2(B) \otimes \mL^{-1}) \oplus
\omega_{\Gamma}^2(B) \oplus \ldots \oplus
(\omega_{\Gamma}^2(B)\otimes \mL^{n-2}).
\end{split}
\end{equation}
In the equations~\eqref{eq.decomp.omega}, the eigensheaves
corresponding to $\chi_j$ are $\omega_{\Gamma} \otimes \mL^j$ and
$\omega_{\Gamma}^2(B) \otimes \mL^j$, respectively.

\subsection{Cyclic quotient singularities, Hirzebruch Jung resolutions
and singular product-quotient surfaces} \label{cyclic.quotient}
Let $n$ and $q$ be natural numbers with $0 < q < n$, $(n,q)=1$ and
let $\zeta$ be a primitive $n$-th root of unity. Let us consider
the action of the cyclic group $\mathbb{Z}_n=\langle \zeta
\rangle$ on $\mathbb{C}^2$  defined by $\zeta \cdot (x,\,
y)=(\zeta x,\, \zeta^q y)$. Then the analytic space
$X_{n,q}=\mathbb{C}^2 / \mathbb{Z}_n$ has a cyclic quotient
singularity of type $\frac{1}{n}(1,q)$, and $X_{n,q} \cong X_{n',
q'}$ if and only if $n=n'$ and either $q=q'$ or $qq' \equiv 1$
(mod $n$). The exceptional divisor on the minimal resolution
$\tilde{X}_{n,q}$ of $X_{n,q}$ is a Hirzebruch-Jung string, that
is to say, a connected union $E=\bigcup_{i=1}^k Z_i$ of smooth
rational curves $Z_1, \ldots, Z_k$ with self-intersection $\leq
-2$, and ordered linearly so that $Z_i Z_{i+1}=1$ for all $i$, and
$Z_iZ_j=0$ if $|i-j| \geq 2$. More precisely, given the continued
fraction
\begin{equation*}
\frac{n}{q}=[b_1,\ldots,b_k]=b_1- \cfrac{1}{b_2 -\cfrac{1}{\dotsb -
\cfrac{1}{\,b_k}}}, \quad b_i\geq 2,
\end{equation*}

\noindent the dual graph of $E$ is  {\setlength{\unitlength}{1.1cm}
\begin{center}
\begin{picture}(1,0.5)
\put(0,0){\circle*{0.2}} \put(1,0){\circle*{0.2}}
\put(0,0){\line(1,0){1}} \put(-0.3,0.2){\scriptsize $-b_1$}
\put(0.70,0.2){\scriptsize $-b_2$} \put(2,0){\circle*{0.2}}
\put(1,0){\line(1,0){0.2}} \put(1.3,0){\line(1,0){0.15}}
\put(1.55,0){\line(1,0){0.15}} \put(1.8,0){\line(1,0){0.2}}
\put(3,0){\circle*{0.2}} \put(2,0){\line(1,0){1}}
\put(1.70,0.2){\scriptsize $-b_{k-1}$} \put(2.70,0.2){\scriptsize
$-b_k$}
\end{picture}         \hspace{2.5cm}
\end{center}}

\vspace{.5cm}
\noindent (cf. \cite[Chapter II]{Lau71}). Notice that a rational
double point of type $A_n$ corresponds to the cyclic quotient
singularity $\frac{1}{n+1}(1,n)$.

\begin{definition} \label{numbers}
Let $x$ be a cyclic quotient singularity of type
$\frac{1}{n}(1,q)$. Then we set
\begin{equation*}
\begin{split}
\mathfrak{h}_x&=2- \frac{2+q+q'}{n}-\sum_{i=1}^k (b_i-2), \\
\mathfrak{e}_x&=k+1-\frac{1}{n}, \\
B_x&= 2 \mathfrak{e}_x - \mathfrak{h}_x = \frac{1}{n} (q + q') +
\sum_{i=1}^k b_i,
\end{split}
\end{equation*}
where $1\leq q' \leq n-1$ is such that $qq' \equiv 1$
$($\textrm{mod} $n)$.
\end{definition}

\begin{definition}\emph{\cite{BP}} \label{def-stand}
We say that a projective surface $S$ is a \emph{product-quotient
surface} if there exists a finite group $G$ acting faithfully on
two smooth projective curves $C_1$ and $C_2$ and diagonally on
their product, so that $S$ is isomorphic to the minimal
desingularization of $X:=(C_1 \times C_2)/G$. The surface $X$ is
called a \emph{singular model of a product-quotient surface}, or
simply a \emph{singular product-quotient surface}.
\end{definition}

From this definition it follows that a singular product quotient
surface contains a finite number of cyclic quotient singularities.

\begin{proposition}[cf. \cite{MP10}, Section 3] \label{invariants-S}
Let $S$ be a product quotient surface, minimal desingularization
of $X=(C_1 \times C_2)/G$. Then the invariants of $S$ are
\begin{itemize}
\item[$\boldsymbol{(i)}$] $K_S^2
=\frac{8(g(C_1)-1)(g(C_2)-1)}{|G|} + \sum \limits_{x \in
\emph{Sing}\; X} \mathfrak{h}_x$. \item[$\boldsymbol{(ii)}$]
$e(S)=\frac{4(g(C_1)-1)(g(C_2)-1)}{|G|}+\sum \limits_{x \in
\emph{Sing}\; X} \mathfrak{e}_x$. \item[$\boldsymbol{(iii)}$]
$q(S)=g(C_1/G)+g(C_2/G)$.
\end{itemize}
\end{proposition}

Set $\Gamma_i:=C_i/G$ and let $g_i \colon C_i \to \Gamma_i$. The
group $G$ acts naturally on the sheaves ${g_i}_* \oo_{C_i}$,
${g_i}_* \omega_{C_i}$, ${g_i}_* \omega_{C_i}^2$. Assuming that
$G$ is \emph{abelian}, we can write the following generalizations
of \eqref{eq.decomp.O} and \eqref{eq.decomp.omega}:

\begin{equation*}
\begin{split}
{g_i}_* \oo_{C_i} & = \bigoplus_{\chi \in \widehat{G}} ({g_i}_*
\oo_{C_i})^{\chi}, \\  {g_i}_* \omega_{C_i} & = \bigoplus_{\chi
\in \widehat{G}} ({g_i}_* \omega_{C_i})^{\chi}, \\   {g_i}_*
\omega_{C_i}^2 & = \bigoplus_{\chi \in \widehat{G}} ({g_i}_*
\omega^2_{C_i})^{\chi},
\end{split}
\end{equation*}
where $(*)^{\chi}$ is the eigensheaf corresponding to the
character $\chi \in \widehat{G}$.

\section{The main construction} \label{sec.basic.construction}

Let us consider two smooth curves $C_1$, $C_2$ of genus $3$, such
that there are two \emph{simple} $\mathbb{Z}_4$-covers $g_i \colon
C_i \to \mathbb{P}^1$, both branched in $4$ points. In the rest of
the paper we write $G:=\mathbb{Z}_4 = \langle \zeta \, | \, \
\zeta^4=1 \rangle$, where $\zeta$ is a primitive fourth root of
unity; we also denote by $H$ the subgroup of $G$ defined by
$H:=\langle \zeta^2 \rangle \cong \mathbb{Z}_2$.

Now set $Z := C_1 \times C_2$ and consider the singular
product-quotient surface
\begin{equation} \label{dia.X}
X:=Z/G,
\end{equation}
which has exactly 16 isolated singular points, corresponding to the
fixed points of the $G$-action on $Z$. Let $\lambda \colon S \to X$
be the minimal resolution of singularities of $X$.

The $G$-cover $g_i$ factors through the double cover $h_i \colon
C_i \to E_i$, where $E_i:=C_i/H$. Note that $E_i$ is an elliptic
curve and that  the singular product-quotient surface
\begin{equation} \label{dia.Y}
Y:=Z/H
\end{equation}
contains sixteen cyclic quotient singularities of type
$\frac{1}{2}(1, \, 1)$, i.e. ordinary double points, as only
singularities. Let us denote by $\mu \colon V \to Y$  the minimal
desingularization of $Y$. We have a commutative diagram
\begin{equation} \label{dia.XY}
\xymatrix{ V \ar[r]^{\mu} &  Y \ar[dd]_s  \ar[rr]^v & &   E_1 \times E_2 \ar[dd]^t \\
  &  & Z  \ar[dr]^g \ar[dl]_p \ar[ul]^r  \ar[ur]_h &  \\
   S \ar[r]^{\lambda} & X \ar[rr]^u & &  \mathbb{P}^1 \times
   \mathbb{P}^1},
\end{equation}
where:
\begin{itemize}
\item $p \colon Z \to X$ and $r \colon Z \to Y$ are the natural
projections, so $s \colon Y \to X$ is a double cover (more
precisely, a $G/H$-cover) branched over the singular points of
$X$; \item $g :=g_1 \times g_2 \colon Z \to \mathbb{P}^1 \times
\mathbb{P}^1$ is a $G \times G$-cover branched on a divisor
$B\subset \mathbb{P}^1 \times \mathbb{P}^1$ of product type and of
bidegree $(4, \, 4)$; \item $h:=h_1 \times h_2 \colon Z \to E_1
\times E_2$ is a $H \times H$-cover branched on a divisor $\Delta
\subset E_1 \times E_2$ of product type and of bidegree $(4,
\,4)$; \item $u \colon X \to \mathbb{P}^1 \times \mathbb{P}^1$ is
a $G$-cover, whose branch locus coincides with $B$; \item $v
\colon Y \to E_1 \times E_2$ is a $H$-cover, whose branch locus
coincides with $\Delta$; \item $t \colon E_1 \times E_2 \to
\mathbb{P}^1 \times \mathbb{P}^1$ is a $G/H \times G/H$-cover
whose branch locus is $B$ and whose ramification locus is
$\Delta$.
\end{itemize}

Let us denote by $B_i$ the branch locus of $g_i \colon C_i \to
\mathbb{P}^1$ and by $\Delta_i$ the branch locus of $h_i \colon
C_i \to E_i$. Both $B_i$ and $\Delta_i$ consist of four points;
clearly
 $B=B_1 \times B_2$ and $\Delta=\Delta_1 \times \Delta_2$. From the
results of Section \ref{sec.prel} we infer that
\begin{itemize}
\item there is a natural action of $G$ on the sheaves ${g_i}_*
\oo_{C_i}$, ${g_i}_* \omega_{C_i}$, ${g_i}_* \omega_{C_i}^2$,
which gives decompositions:
\begin{equation} \label{eq.decomp.f}
\begin{split}
{g_i}_* \oo_{C_i}&=\mathcal{O}_{\mathbb{P}^1} \oplus
\mathcal{M}_i^{-1} \oplus \mathcal{M}_i^{-2}
\oplus  \mathcal{M}_i^{-3}; \\
{g_i}_* \omega_{C_i}&=\omega_{\mathbb{P}^1} \oplus
(\omega_{\mathbb{P}^1}\otimes \mathcal{M}_i) \oplus
(\omega_{\mathbb{P}^1} \otimes \mathcal{M}_i^2)
\oplus  (\omega_{\mathbb{P}^1}\otimes \mathcal{M}_i^3); \\
{g_i}_* \omega_{C_i}^2&=\omega_{\mathbb{P}^1}^2(B_i) \oplus
(\omega_{\mathbb{P}^1}^2(B_i) \otimes \mathcal{M}_i)\oplus
(\omega_{\mathbb{P}^1}^2(B_i) \otimes \mathcal{M}_i^2) \\
& \oplus
(\omega_{\mathbb{P}^1}^2(B_i) \otimes \mathcal{M}_i^{-1}),
\end{split}
\end{equation}
where $\mathcal{M}_i=\oo_{\mP^1}(1)$. Left to right, the direct
summands are the four eigensheaves corresponding to the four
characters $\chi_0$, $\chi_1$, $\chi_2$, $\chi_3$ of $G;$ \item
there is a natural action of $H$ on the sheaves ${h_i}_*
\oo_{C_i}$, ${h_i}_* \omega_{C_i}$, ${h_i}_* \omega_{C_i}^2$,
which gives decompositions:
\begin{equation} \label{eq.decomp.g}
\begin{split}
{h_i}_* \oo_{C_i}&=\oo_{E_i} \oplus \mathcal{L}_i^{-1}, \\
{h_i}_*\omega_{C_i}&=\omega_{E_i} \oplus (\omega_{E_i} \otimes \mathcal{L}_i), \\
{h_i}_*\omega_{C_i}^2&=\omega_{E_i}^2(\Delta_i) \oplus
(\omega_{E_i}^2(\Delta_i) \otimes \mL_i^{-1}),
\end{split}
\end{equation}
where $\mathcal{L}_i$ is a line bundle of degree $2$ on $C_i$ such
that $\mathcal{L}_i^{2}= \oo_{E_i}(\Delta_i)$. Left to right, the
direct summands correspond to the invariant and anti-invariant
eigensheaves for the $H$-action, respectively.
\end{itemize}

\section{Deformations of the singular product-quotient surface $Y=Z/H$} \label{sec.example.Y}

Let us consider again the surface $Y=Z/H$ defined in Section
\ref{sec.basic.construction}, together with its minimal
desingularization $\mu \colon V \to Y$. As we remarked in the
previous section, we have
\begin{equation*}
\textrm{Sing}(Y) = 16 \times \frac{1}{2}(1, \, 1).
\end{equation*}

\begin{proposition} \label{prop.inv.T}
$V$ is a minimal surface of general type whose invariants are
\begin{equation*}
\begin{split}
p_g(V)&=5, \quad q(V)=2, \quad K_V^2=16, \\
 h^1(\Theta_V)& =24, \quad h^2(\Theta_V)=16.
\end{split}
\end{equation*}
\end{proposition}
\begin{proof}
The invariants $p_g(V)$, $q(V)$, $K_V^2$ can be computed by using
Proposition \ref{invariants-S}. Since $p_g(V)
>0$ and $K_V^2 >0$, it follows that
$V$ is a surface of general type.
 Let us denote by $H^0(\ast)^+$ and $H^0(\ast)^-$ the spaces
of invariant and anti-invariant sections for the $H$-action and by
$h^0(\ast)^+$ and $h^0(\ast)^-$ their dimensions. Since $Y$ has only
rational double points, K$\ddot{\textrm{u}}$nneth formula and the
third equality in \eqref{eq.decomp.g} give
\begin{equation*}
\begin{split}
&H^0(\omega_V^2)=H^0(\omega_Y^2)=H^0(\omega_Z^2)^+=H^0(\omega_{C_1}^2
\boxtimes \omega_{C_2}^2)^+
\\
&=(H^0({h_1}_* \omega_{C_1}^2)^+ \otimes  H^0({h_2}_*
\omega_{C_2}^2)^+ )\oplus (H^0({h_1}_* \omega_{C_1}^2)^- \otimes
H^0({h_2}_* \omega_{C_2}^2)^-) \\
& \cong \mathbb{C}^{20}.
\end{split}
\end{equation*}
This shows that $h^0(\omega_V^2)=K_V^2 + \chi(\oo_V)$, hence $V$
is a minimal model.

Since $Y$ is a normal surface, \cite[Proposition 1.2]{BW74} gives
 $\mu_*\Theta_V= \Theta_Y$. Therefore the argument in \cite[Section 1]{BW74} or \cite[p.
299]{Cat89} shows that there are two isomorphisms
\begin{equation} \label{H1}
H^1(\Theta_V) \cong H^1(\Theta_Y) \oplus H^1_E(\Theta_V), \quad
H^2(\Theta_V) \cong H^2(\Theta_Y),
\end{equation}
where $H^1_E(\Theta_V)$ denotes the local cohomology with support
on the exceptional divisor $E \subset V$.

By the second isomorphism in \eqref{H1}, we have
\begin{equation} \label{eq.decompH2}
H^2(\Theta_V)^* \cong H^2(\Theta_Y)^\ast = H^0(\Omega_Z^1 \otimes
\Omega_Z^2)^+=T_1 \oplus T_2 \oplus T_3 \oplus T_4,
\end{equation}
where
\begin{equation} \label{eq.decompH2-1}
\begin{split}
T_1 & =  H^0({h_1}_*\omega_{C_1}^2)^+ \otimes H^0({h_2}_*
\omega_{C_2})^+=H^0(\omega_{E_1}^2(\Delta_1)) \otimes
H^0(\omega_{E_2}), \\
T_2 & = H^0({h_1}_* \omega_{C_1})^+ \otimes H^0({h_2}_*
\omega_{C_2}^2)^+ = H^0(\omega_{E_1}) \otimes
H^0(\omega_{E_2}^2(\Delta_2)),\\
T_3 & = H^0({h_1}_* \omega_{C_1}^2)^- \otimes H^0({h_2}_*
\omega_{C_2})^-\\
&=H^0(\omega_{E_1}^2(\Delta_1) \otimes
 \mathcal{L}_1^{-1})\otimes
H^0(\omega_{E_2} \otimes \mathcal{L}_2), \\
T_4&=H^0({h_1}_* \omega_{C_1})^- \otimes H^0({h_2}_*
\omega_{C_2}^2)^- \\
& = H^0(\omega_{E_1} \otimes \mathcal{L}_1)
\otimes H^0(\omega_{E_2}^2(\Delta_2) \otimes \mathcal{L}_2^{-1}).
\end{split}
\end{equation}
Since $\dim T_i=4$ for all $i \in \{1,\,2,\,3, \, 4\}$, we infer
$h^2(\Theta_V)=h^2(\Theta_Y)=16$. By Riemann-Roch we have
$h^1(\Theta_V) - h^2(\Theta_V)=10 \chi(\oo_V) - 2 K_V^2=8$, so it
follows $h^1(\Theta_V)=24$.
\end{proof}

\begin{corollary} \label{cor.Theta.Y}
We have
\begin{equation*}
h^1(\Theta_Y)=8, \quad h^2(\Theta_Y)=16.
\end{equation*}
\end{corollary}
\begin{proof}
Since $h^2(\Theta_Y)=h^2(\Theta_V)$, the first equality follows from
Proposition \ref{prop.inv.T}. Furthermore, $E$ is the disjoint union
of sixteen $(-2)$-curves, hence \cite[Section 1]{BW74} implies
$H^1_E(\Theta_V) \cong \mathbb{C}^{16}$. Using $h^1(\Theta_V)=24$
and the first isomorphism in \eqref{H1} we obtain $h^1(\Theta_Y)=8$,
which completes the proof.
\end{proof}

By using the local-to-global spectral sequence of $\EE xt$-sheaves
we obtain an exact sequence

\begin{equation} \label{seq.sing.Y}
0 \to H^1(\Theta_Y) \lr \textrm{Ext}^1(\Omega^1_Y, \, \oo_Y) \lr
\mathcal{T}^1_Y \stackrel{\textrm{ob}_Y}{\lr} H^2(\Theta_Y),
\end{equation}
where $\mathcal{T}^1_Y:=H^0(\EE xt^1(\Omega^1_Y, \, \oo_Y))$.
Notice that $\mathcal{T}^1_Y$ is a skyscraper sheaf supported on
the sixteen nodes of $Y$, hence $\textrm{ob}_Y$ is a linear map
$$\textrm{ob}_Y \colon \mathbb{C}^{16} \to \mathbb{C}^{16}.$$ Thus
its kernel and its cokernel have the same dimension.

\begin{remark} \label{rem.dim=18}
The branch locus $\Delta$ of $v \colon Y \to E_1 \times E_2$ is a
polarization of type $(4, \, 4)$ on the abelian surface $E_1 \times
E_2$, in particular $h^0(\Delta)=16$. Since polarized abelian
surfaces form a $3$-dimensional family, it follows that the
deformation space $\emph{Def}(Y)$ has dimension at least $18$.
Therefore we have
\begin{equation*}
\dim \emph{Ext}^1(\Omega_Y^1, \, \oo_Y)=\dim T_{[Y]} \emph{Def}(Y)
\geq \dim_{[Y]} \emph{Def}(Y) \geq 18.
\end{equation*}
\end{remark}

\begin{proposition} \label{prop.ker.ob}
We have
\begin{equation*}
\dim \ker \emph{ob}_Y = \dim \emph{coker}\, \emph{ob}_Y =10.
\end{equation*}
\end{proposition}
\begin{proof}
Notice that Remark \ref{rem.dim=18} only gives $\dim(\ker
\textrm{ob}_Y)\ge 10$. In order to prove equality, we apply an
argument used in \cite[Section 2]{Cat89}.

Let us consider the dual map ${\rm ob}_Y^\ast: H^2(\Theta_Y)^\ast
\to (\mathcal{T}^1_Y)^{\ast}$. We set
\begin{equation*}
\begin{split}
\Delta_1 &=  d_1'+d_2'+d_3'+d_4' \\
\Delta_2  &=d_1''+d_2''+d_3''+d_4''
\end{split}
\end{equation*}
and we choose local coordinates $(x, \, y)$ in $Z$ vanishing at
$(d_i',\, d_j'')$. Then the action of $H$ with respect to these
coordinates is given by $(x, \,y) \to (-x, \, -y)$.

By \cite{Cat89} we have an isomorphism
$(\mathcal{T}^1_Y)^*=(r_*\Omega^1_Z)^+/\Omega^1_Y$, therefore
$\textrm{ob}_Y^*$ can be seen as a map
\begin{equation*}
\textrm{ob}_Y^* \colon H^0(\Omega_Z^1 \otimes \Omega_Z^2)^+ \to
(r_*\Omega^1_Z)^+/\Omega^1_Y.
\end{equation*}

Near any of the ordinary double points of $Y$, the sheaf
$(r_*\Omega^1_Z)^+$ is locally generated by $xdx$, $xdy$, $ydx$,
$ydy$, whereas $\Omega^1_Y$ is locally generated by $d(x^2)$,
$d(xy)$,  $d(y^2)$; then $(r_*\Omega^1_Z)^+/\Omega^1_Y$ is locally
generated by $xdy-ydx$,
 cf. \cite[Lemma 2.11]{Cat89}.

Looking at \eqref{eq.decompH2-1} and making straightforward
computations, one checks that
\begin{itemize}
\item the summand $T_1$ contributes expressions of type $\alpha_1
\beta_1 ydx \otimes (dx \wedge dy)$; \item the summand $T_2$
contributes expressions of type $\alpha_2 \beta_2 xdy \otimes
(dx\wedge dy)$; \item the summand $T_3$ contributes expressions of
type $\alpha_3 \beta_3 xdx \otimes (dx \wedge dy)$; \item the
summand $T_4$ contributes expressions of type $\alpha_4 \beta_4
ydy \otimes (dx \wedge dy)$,
\end{itemize}
where $\alpha_i=\alpha_i(x^2)$ and $\beta_i=\beta_i(y^2)$ are
pullbacks of local functions on $E_i$.

Since in the $\oo_Y$-module $(r_*\Omega^1_Z)^+/\Omega^1_Y$ we have
the relations
$$1/2(xdy-ydx)=xdy=-ydx\,\, \text{and} \,\, xdx=ydy=0,$$ it follows
that the restriction of $\textrm{ob}_Y^*$ to the subspace $T_3
\oplus T_4$ is zero, whereas the restriction of $\textrm{ob}_Y^*$
to the subspace $T_1 \oplus T_2$ can be identified, up to a
 multiplicative constant, with the map
\begin{equation*}
\begin{split}
\phi \colon H^0 & (\omega_{E_1}^2(\Delta_1)) \oplus
H^0(\omega_{E_2}^2(\Delta_2))
\to \bigoplus_{i,j=1}^4 \mathbb{C}_{ij}, \\
\phi & (\sigma \oplus \tau)=
\bigoplus_{i,j=1}^4(\textrm{val}_{d_i'} (\sigma)
-\textrm{val}_{d_j''}(\tau)).
\end{split}
\end{equation*}

Here the valuation maps $\textrm{val}_{d_i'}$ and
$\textrm{val}_{d_j''}$ are defined, as usual, by the short exact
sequences
\begin{equation} \label{eq.valuation}
\begin{split}
 0 & \to
H^0(\omega_{E_1}^2)\to H^0(\omega_{E_1}^2(\Delta_1))\stackrel{\oplus
\textrm{val}_{d_i'}}{\lr}H^0(N_{\Delta_1}) \cong \oplus_{i=1}^4 \mathbb C_i, \\
0 & \to H^0(\omega_{E_2}^2)\to
H^0(\omega_{E_2}^2(\Delta_2))\stackrel{\oplus
\textrm{val}_{d_j''}}{\lr} H^0(N_{\Delta_2}) \cong \oplus_{j=1}^4
\mathbb C_j.
\end{split}
\end{equation}

Therefore we obtain
\begin{equation} \label{ker-p}
\begin{split}
\ker \phi =\{\sigma \oplus \tau \, | \,
&\textrm{val}_{d_1'}(\sigma)=\textrm{val}_{d_2'}(\sigma)=\textrm{val}_{d_3'}(\sigma)=\textrm{val}_{d_4'}(\sigma)
\\
=
&\textrm{val}_{d_1''}(\tau)=\textrm{val}_{d_2''}(\tau)=\textrm{val}_{d_3''}(\tau)=\textrm{val}_{d_4''}(\tau)
\}.
\end{split}
\end{equation}

As $E_i$ is an elliptic curve, we have $\omega^2_{E_i}=\omega_{E_i}$
and so $\eqref{eq.valuation}$ are the standard residue sequences for
meromorphic $1$-forms. By the Residue Theorem we get
\begin{equation*}
\sum_{i=1}^4 \textrm{val}_{d_i'}(\sigma)=\sum_{j=1}^4
\textrm{val}_{d_j''}(\tau)=0,
\end{equation*}
hence \eqref{ker-p} implies that $\sigma \oplus \tau \in \ker
\phi$ if and only if
$\textrm{val}_{d_i'}(\sigma)=\textrm{val}_{d_j''}(\tau)=0$ for all
pairs $(i,\,j)$. This yields $\ker \phi = H^0(\omega^2_{E_1})
\oplus H^0(\omega^2_{E_2}) \cong \mathbb{C} \oplus \mathbb{C}$.

Then $\ker \textrm{ob}^*=\ker \phi \oplus T_3 \oplus T_4 \cong
\mathbb{C}^{10}$, hence $\dim {\rm coker}\, {\rm ob}_Y = 10$ and
we are done.
\end{proof}

\begin{corollary} \label{cor.ext1}

We have
\begin{equation*}
\dim \emph{Ext}^1(\Omega_Y^1, \, \oo_Y)=18.
\end{equation*}
\end{corollary}
\begin{proof}
Immediate from Corollary \ref{cor.Theta.Y}, Proposition
\ref{prop.ker.ob} and exact sequence
 \eqref{seq.sing.Y}.
\end{proof}

\begin{proposition} \label{prop.def(Y)}
The following holds:
\begin{itemize}
\item[$\boldsymbol{(i)}$] $\emph{Def}(Y)$ is smooth at $[Y]$, of dimension $18;$
\item[$\boldsymbol{(ii)}$] $\emph{ESDef}(Y)$ is smooth at $[Y]$, of
dimension $8$.
\end{itemize}
\end{proposition}
\begin{proof}
By Remark \ref{rem.dim=18} and Corollary \ref{cor.ext1} we have
\begin{equation*}
18 = \dim \textrm{Ext}^1(\Omega_Y^1, \, \oo_Y)=\dim T_{[Y]}
\textrm{Def}(Y) \geq \dim_{[Y]} \textrm{Def}(Y) \geq 18,
\end{equation*}
which proves $\boldsymbol{(i)}$.

On the other hand, if we move the branch loci $B_i \subset E_i$
the curve $\Delta \subset E_1 \times E_2$ remains of product type,
so in this way we obtain a $8$-dimensional family of
\emph{equisingular} deformations of $Y$; therefore the
equisingular deformation space $\textrm{ESDef}(Y)$ has dimension
at least $8$, and  by Corollary \ref{cor.Theta.Y} we have
\begin{equation*}
8=\dim H^1(\Theta_Y)=\dim T_{[Y]} \textrm{ESDef}(Y) \geq \dim_{[Y]}
\textrm{ESDef}(Y) \geq 8.
\end{equation*}
This proves $\boldsymbol{(ii)}$.
\end{proof}

Summing up, Proposition \ref{prop.def(Y)} shows that the
deformations of $Y$ are unobstructed and that they are all obtained
by deforming the pair $(A,
 \,\Delta)$, where $A$ is an abelian surface and $\Delta$ a
polarization of type $(4,\,4)$. In particular, all the
deformations preserve the action of $H$. Moreover, the
equisingular deformations of $Y$ are also unobstructed and are
obtained by taking as $A$ the product of two elliptic curves and
by choosing the polarization $\Delta$ of product type.

\begin{remark} \label{rem.no.ind.beh.V}
Since $Y$ has only rational double points, by \emph{\cite{BW74}}
the dimension of $\emph{Def}(Y)$ equals the dimension of
$\emph{Def}(V)$. Then
\begin{equation*}
24= h^1(\Theta_V)= \dim T_{[V]} \emph{Def}(V) > \dim_{[V]}
\emph{Def}(V)=18,
\end{equation*}
that is $\emph{Def}(V)$ is \emph{singular} at $[V]$. By
\emph{\cite[Theorem 3.7]{BW74}}, this means that the sixteen
$(-2)$-curves of $V$ do not have independent behavior in
deformations.
\end{remark}

\section{Deformations of the singular product-quotient surface $X=Z/G$} \label{sec.example.X}

Let us consider now the surface $X=Z/G$ defined in Section
\ref{sec.basic.construction} and its minimal resolution of
singularities $\lambda \colon S\to X$. We must analyze several
cases, according to the type of quotient singularities that $X$
contains.

Throughout this section we set $Q:=\mathbb{P}^1 \times
\mathbb{P}^1$ and we denote by $\mathcal{O}_Q(a, \, b)$ the line
bundle of bidegree $(a, \, b)$ on $Q$.

The following exact sequence is the analogue of \eqref{seq.sing.Y}:
\begin{equation} \label{seq.sing.X}
0 \to H^1(\Theta_X) \lr \textrm{Ext}^1(\Omega^1_X, \, \oo_X) \lr
\mathcal{T}^1_X \stackrel{\textrm{ob}_X}{\lr} H^2(\Theta_X).
\end{equation}

\subsection{Example where $\textrm{Sing}(X)= 16\times\frac{1}{4}(1, \, 3)$}

Assume that, locally around each of the fixed points, the action
of $G= \langle \zeta \, | \, \zeta^4=1 \rangle$  is given by
$\zeta \cdot (x, \, y)=(\zeta x, \, \zeta^{-1} y)$. Therefore,
\begin{equation*}
\textrm{Sing}(X) = 16 \times \frac{1}{4}(1,3).
\end{equation*}

In this case $X$ contains only rational double points and we
obtain
\begin{equation*}
p_g(S)=5, \quad q(S)=0, \quad K_S^2=8.
\end{equation*}

\begin{proposition} \label{prop X minimal}
$S$ is a minimal surface of general type.
\end{proposition}
\begin{proof}
$S$ is of general type because $p_g(S)>0$ and $K_S^2 >0$. Since the
action of $G$ is twisted on the second factor and $X$ has only
rational double points, the K$\ddot{\textrm{u}}$nneth formula and
the third equality in \eqref{eq.decomp.f} give
\begin{equation*}
\begin{split}
H^0(\omega_S^2)&=H^0(\omega_X^2)=H^0(\omega_Z^2)^G=H^0(\omega_{C_1}^2
\boxtimes \omega_{C_2}^2)^G
\\
&=\bigoplus_{\chi \in \widehat{G}}(H^0({g_1}_*
\omega_{C_1}^2)^{\chi} \otimes H^0({g_2}_* \omega_{C_2}^2)^{\chi}
)=\mathbb{C}^{14}.
\end{split}
\end{equation*}
This shows that $h^0(\omega_S^2)=K_S^2 + \chi(\oo_S)$, hence $S$
is a minimal surface.
\end{proof}

\begin{proposition} \label{prop.example.1}
The following holds:
\begin{itemize}
\item[$\boldsymbol{(i)}$] $\emph{ob}_X$ is surjective;
\item[$\boldsymbol{(ii)}$] $h^1(\Theta_X)=2, \quad
h^2(\Theta_X)=6, \quad h^1(\Theta_S)=50, \quad h^2(\Theta_S)=6$.
\item[$\boldsymbol{(iii)}$] $\emph{ESDef}(X)$ is smooth at $[X]$,
of dimension $2$.
\end{itemize}
\end{proposition}
\begin{proof}
$\boldsymbol{(i)}$  Let us consider the dual map ${\rm ob}_X^\ast:
H^2(\Theta_X)^\ast \to (\mathcal{T}^1_X)^\ast$. By Grothendieck
duality (see \cite[Chapter I]{AK70}) and K$\ddot{\textrm{u}}$nneth
formula we obtain
\begin{equation} \label{U}
\begin{split}
H^2(\Theta_X)^\ast & = H^0(\Omega_Z^1 \otimes \Omega_Z^2)^G \\
& =  \bigoplus_{\chi \in \widehat G} \big[
(H^0({g_1}_*\omega_{C_1})^\chi \otimes H^0({g_2}_*
\omega_{C_2}^2)^\chi)  \\
& \quad \quad \oplus  (H^0({g_1}_* \omega_{C_1}^2)^\chi
\otimes H^0({g_2}_* \omega_{C_2})^\chi) \big]  \\
& = U_1 \oplus U_2, \,\, \text{where} \\
& U_1=H^0(\omega_{\mathbb{P}^1} \otimes \mathcal{M}_1^2 )\otimes
H^0(\omega_{\mathbb{P}^1}^2(B_2) \otimes \mathcal{M}_2^2), \\
& U_2=H^0(\omega_{\mathbb{P}^1}^2(B_1) \otimes \mathcal{M}_1^2
)\otimes H^0(\omega_{\mathbb{P}^1} \otimes \mathcal{M}_2^2).
\end{split}
\end{equation}
This yields $h^2(\Theta_X)=6$ and so $h^2(\Theta_S)=6$. Now we set
\begin{equation*}
\begin{split}
B_1 & = b_1'+b_2'+b_3'+b_4'\\
B_2 & = b_1''+b_2''+b_3''+b_4''
\end{split}
\end{equation*}
and we choose local coordinates $(x, \, y)$ in $Z$ vanishing at
$(b_i', \,b_j'')$. As in Section \ref{sec.example.Y}, we can
interpret  $\textrm{ob}_X^*$ as a map
\begin{equation*}
\textrm{ob}_X^* \colon H^0(\Omega_Z^1 \otimes \Omega_Z^2)^G \to
(p_*\Omega^1_Z)^G/\Omega^1_X,
\end{equation*}
where $(p_*\Omega^1_Z)^G/\Omega^1_X$ is a skyscraper sheaf
supported on the singular points of $X$ and locally generated by
$x^iy^{i+1}dx-y^ix^{i+1}dy$, for $i=0,1,2$, see \cite{Cat89}.

A straightforward local computation shows that the summand $U_1$
in \eqref{U} contributes expressions of the form $\alpha_1 \beta_1
xdy \otimes (dx \wedge dy)$ whereas the summand $U_2$ contributes
expressions of the form $\alpha_2 \beta_2 ydx \otimes (dx \wedge
dy)$, where $\alpha_i=\alpha_i(x^2)$ and $\beta_i=\beta_i(y^2)$
are pullbacks of local functions on $\mathbb{P}^1$. Therefore the
map $\textrm{ob}_X^*$ can be identified, up to a multiplicative
constant, with
\begin{equation*}
\begin{split}
\phi \colon & H^0(\omega_{\mathbb{P}^1}^2(B_1)  \otimes
\mathcal{M}_1^2) \oplus H^0(\omega_{\mathbb{P}^1}^2(B_2) \otimes
\mathcal{M}_2^2) \\
& \to \bigoplus_{i,j=1}^4 \mathbb{C}_{ij}  \subset
\bigoplus_{i,j=1}^{4} \mathbb{C}_{ij}^{\oplus 3} \cong
 (\mathcal{T}_X^1)^{\ast} \\
  & \phi (\sigma \oplus \tau)=
\bigoplus_{i,j=1}^4(\textrm{val}_{b_i'} (\sigma)
-\textrm{val}_{b_j''}(\tau)),
\end{split}
\end{equation*}
where the valuation maps are defined as in Section
\ref{sec.example.Y}. Hence we obtain
\begin{equation} \label{ker-p-2}
\begin{split}
\ker \phi =\{\sigma \oplus \tau \, | \,
&\textrm{val}_{b_1'}(\sigma)=\textrm{val}_{b_2'}(\sigma)=\textrm{val}_{b_3'}(\sigma)=\textrm{val}_{b_4'}(\sigma)
\\
=
&\textrm{val}_{b_1''}(\tau)=\textrm{val}_{b_2''}(\tau)=\textrm{val}_{b_3''}(\tau)=\textrm{val}_{b_4''}(\tau)
\}.
\end{split}
\end{equation}
On the other hand, the valuation map
$H^0(\omega_{\mathbb{P}^1}^2(B_i) \otimes \mathcal{M}_i^2) \to
H^0(N_{B_i})$ can be identified with the residue map
$H^0(\omega_{\mathbb{P}^1}(B_i)) \to H^0(N_{B_i})$ via the
isomorphism $H^0(\omega_{\mathbb{P}^1}^2(B_i) \otimes
\mathcal{M}_i^2) \cong H^0(\omega_{\mathbb{P}^1}(B_i))$. By the
Residue Theorem  we have
\begin{equation*}
 \sum_{i=1}^4 \textrm{val}_{b_i'}(\sigma)=\sum_{j=1}^4
\textrm{val}_{b_j''}(\tau)=0,
\end{equation*}
so \eqref{ker-p-2} implies that $\sigma \oplus \tau \in \ker \phi$
if and only if
$\textrm{val}_{b_i'}(\sigma)=\textrm{val}_{b_j''}(\tau)=0$ for all
pairs $(i, \,j)$. But there are no non-zero holomorphic $1$-forms on
$\mathbb{P}^1$, so $\ker \phi=0$ and $\textrm{ob}_X^*$ is injective.
Therefore the obstruction map ${\rm ob}_X$ is surjective.

$\boldsymbol{(ii)}$ Let us denote by $F \subset S$ the exceptional
divisor of $\lambda \colon S \to X$. Since $S$ has only rational
double points, we have
\begin{equation*}
H^1(\Theta_S) \cong H^1(\Theta_X) \oplus H^1_F(\Theta_S), \quad
H^2(\Theta_S) \cong H^2(\Theta_X).
\end{equation*}
By Riemann-Roch theorem we obtain
\begin{equation*}
h^1(\Theta_S)-h^2(\Theta_S) = 10 \chi(\oo_S) - 2 K_S^2=44,
\end{equation*}
then $h^1(\Theta_S)=50$ since we have shown that
$h^2(\Theta_S)=6$, see part $\boldsymbol{(i)}$. Being $F$ the
union of sixteen disjoint $A_3$-cycles, we have $H^1_F(\Theta_S)
\cong \mathbb{C}^{16 \cdot 3}=\mathbb{C}^{48}$. Therefore
$h^1(\Theta_X)=2$.

$\boldsymbol{(iii)}$ The cover $u \colon X \to Q$ is a simple
$G$-cover branched on the divisor $B=B_1 \times B_2$, which has
bidegree $(4, \, 4)$. By varying the branch loci $B_i \subset
\mathbb{P}^1$ we obtain a $2$-dimensional family of equisingular
deformations of $X$. Then
\begin{equation*}
2=\dim H^1(\Theta_X)=\dim T_{[X]} \textrm{ESDef}(X) \geq \dim_{[X]}
\textrm{ESDef} (X) \geq 2,
\end{equation*}
which implies the claim.
\end{proof}

\begin{proposition} \label{prop.defX}
The general deformation of the surface $X$ is a canonically
embedded, smooth complete intersection $S_{2,4}$ of type $(2,\,4)$
in $\mathbb{P}^4$.
\end{proposition}
\begin{proof}
By \cite[Proposition 6.2]{Cat97} it is sufficient to check that
the canonical map $\phi_{K} \colon X \to \mathbb{P}^4$ is a
birational morphism onto its image. Since $X$ has only Rational
Double Points and $u \colon X \to Q$ is a \emph{simple} $G$-cover,
Hurwitz formula yields $K_X = u^* \mathcal{O}_Q(1,\,1)$; but
 $|\mathcal{O}_Q(1,\,1)|$ is
base-point free, so $|K_X|$ is also base-point free and $\phi_K$
is a morphism.

It remains to show that $\phi_K$ separates two general points $x$,
$y$ on $X$. The decomposition of $u_* \omega_X$ with respect to
the $G$-action is
\begin{equation*}
u_* \omega_X = \omega_Q \oplus (\omega_Q \otimes L) \oplus
(\omega_Q \otimes L^2) \oplus (\omega_Q \otimes L^3),
\end{equation*}
where $L=\mathcal{O}_Q(1,\,1)$ and $\omega_Q \otimes L^i$ is the
eigensheaf corresponding to the character $\chi_i$. Therefore we
obtain
\begin{equation*}
H^0(u_*\omega_X)=H^0(\omega_Q \otimes L^2) \oplus H^0(\omega_Q
\otimes L^3).
\end{equation*}
Now let $\{\tau \}$ be a basis of $H^0(\omega_Q \otimes
L^2)=H^0(\mathcal{O}_Q)$ and let $\{\sigma_1, \, \sigma_2,
 \, \sigma_3, \, \sigma_4\}$ be a basis of $H^0(\omega_Q \otimes
L^3)=H^0(\mathcal{O}_Q(1,\,1))$. The four sections $\{\sigma_i\}$
provide an embedding $Q \hookrightarrow \mathbb{P}^3$, hence
$\phi_K$ separates pairs of points which belong to the same fibre of
$u\colon X \to Q$. Now let $x$, $y$ be two points in the same
(general) fibre of $u$. Then there exists $1 \leq a \leq 3$ such
that $y= \zeta^a \cdot x$. Then
\begin{equation*}
\sigma_i(y)=\zeta^a\sigma_i(x), \quad \tau(y)=\zeta^{2a} \tau(x),
\end{equation*}
that is
\begin{equation*}
\begin{split}
\phi_K(y)&=[\sigma_1(y)\colon \sigma_2(y) \colon \sigma_3(y)
\colon \sigma_4(y) \colon \tau(y)]\\
&=[\sigma_1(x)\colon \sigma_2(x)
\colon \sigma_3(x) \colon \sigma_4(x) \colon \zeta^a \tau(x)] \\
& \neq [\sigma_1(x)\colon \sigma_2(x) \colon \sigma_3(x) \colon
\sigma_4(x) \colon \tau(x)] = \phi_K(x).
\end{split}
\end{equation*}

Therefore $\phi_K$ also separates general pairs of points lying in
the same fibre of $u\colon X \to Q$ and we are done.
\end{proof}

Now we can prove the following
\begin{proposition} \label{dim44}
$\emph{Def}(X)$ is smooth at $[X]$, of dimension \emph{44}.
\end{proposition}
\begin{proof}
By using Proposition \ref{prop.example.1} and exact sequence
\eqref{seq.sing.X} we obtain
\begin{equation} \label{ExtX}
\dim T_{[X]} \textrm{Def}(X)= \dim \textrm{Ext}^1(\Omega^1_X, \,
\oo_X)=44.
\end{equation}
On the other hand, by  \cite[Chapter 3]{Se06} one knows that
$\textrm{Def}(S_{2,4})$ is smooth, of dimension
\begin{equation*}
h^0(N_{S_{2,4}/\mathbb{P}^4})- \dim
\textrm{Aut}(\mathbb{P}^4)=h^0(\mathcal{O}_{S_{2,4}}(2))+h^0(\mathcal{O}_{S_{2,4}}(4))-24=44.
\end{equation*}
Equality \eqref{ExtX} and Proposition \ref{prop.defX} yield
\begin{equation} \label{eq.def-2.4}
%\begin{split}
 44 = \dim T_{[X]} \textrm{Def}(X) \geq \dim_{[X]} \textrm{Def}(X)
  = \dim_{[S_{2,4}]}
\textrm{Def}(S_{2,4})=44,
%\end{split}
\end{equation}
so we are done.
\end{proof}

\begin{remark}
Since $X$ has only rational double points, by \emph{{\cite{BW74}}}
the dimension of $\emph{Def}(X)$ equals the dimension of
$\emph{Def}(S)$. So we infer
\begin{equation*}
50= h^1(\Theta_S)= \dim T_{[S]} \emph{Def}(S) > \dim_{[S]}
\emph{Def}(S)=44,
\end{equation*}
that is $\emph{Def}(S)$ is \emph{singular} at $[S]$. By
\emph{\cite[Theorem 3.7]{BW74}}, this means that the sixteen
$A_3$-cycles of $S$ do not have independent behavior in
deformations.
\end{remark}

Proposition \ref{prop.defX} in particular shows that the general
deformation of $X$ does not preserve the $G$-action. Now we want
to consider some particular deformations that preserve the
quadruple cover $u \colon X \to Q$. According to \cite{Pa91} we
call them \emph{natural deformations}, and we freely follow the
notation of that paper everywhere. The building data of any
totally ramified $G$-cover $u \colon X \to Q$ are
\begin{equation} \label{building}
\begin{split}
4L_{\chi_1} & =  3D_{G, \chi_3}+D_{G, \chi_1}\\
2L_{\chi_2} & =  D_{G, \chi_1}+D_{G, \chi_3}\\
4L_{\chi_3} & =  D_{G, \chi_3}+3D_{G, \chi_1},
\end{split}
\end{equation}
see \cite[Proposition 2.1]{Pa91}. The $G$-cover $u \colon X \to Q$
defines a natural embedding $i$ of $X$ into the total space of the
vector bundle $W=\bigoplus_{\chi \in \widehat{G} \setminus
\{\chi_0\} } V(L_\chi^{-1})$. If $w_\chi$ is a local coordinate on
$V(L^{-1}_{\chi})$ on an open set $U$ and $\sigma_{G, \psi}$ is a
local equation for $D_{G, \, \psi}$ on $U$, then $i(X)$ is defined
by the equations
\begin{equation} \label{eq.embed.vec}
w_\chi w_{\chi'}=\bigg( \prod_{\psi \in \{\chi_1,
\chi_3\}}(\sigma_{G, \psi})^{\epsilon^{G, \psi}_{\chi, \chi'}}
\bigg)w_{\chi\chi'}
\end{equation}
and the covering map is given by the composition $\pi \circ i$,
where $\pi \colon W \to Q$ is the projection. Moreover, the
integers $\epsilon^{G, \psi}_{\chi, \chi'}$ can be easily computed
by using \cite[p. 196]{Pa91}:
\begin{equation} \label{eq.epsilon}
\begin{array}{lllll}
\epsilon^{G, \chi_1}_{\chi_0, \chi_0}=0, & \epsilon^{G,
\chi_1}_{\chi_0, \chi_1}=0, & \epsilon^{G, \chi_1}_{\chi_0,
\chi_2}=0, & \epsilon^{G, \chi_1}_{\chi_0, \chi_3}=0, &
\epsilon^{G, \chi_1}_{\chi_1, \chi_1}=0, \\
\epsilon^{G, \chi_1}_{\chi_1, \chi_2}=0, & \epsilon^{G,
\chi_1}_{\chi_1, \chi_3}=1, & \epsilon^{G, \chi_1}_{\chi_2,
\chi_2}=1, & \epsilon^{G, \chi_1}_{\chi_2, \chi_3}=1, &
\epsilon^{G, \chi_1}_{\chi_3, \chi_3}=1, \\
\epsilon^{G, \chi_3}_{\chi_0, \chi_0}=0, & \epsilon^{G,
\chi_3}_{\chi_0, \chi_1}=0, & \epsilon^{G, \chi_3}_{\chi_0,
\chi_2}=0, & \epsilon^{G, \chi_3}_{\chi_0, \chi_3}=0, &
\epsilon^{G, \chi_3}_{\chi_1, \chi_1}=1, \\
\epsilon^{G, \chi_3}_{\chi_1, \chi_2}=1, & \epsilon^{G,
\chi_3}_{\chi_1, \chi_3}=1, & \epsilon^{G, \chi_3}_{\chi_2,
\chi_2}=1, & \epsilon^{G, \chi_3}_{\chi_2, \chi_3}=0, &
\epsilon^{G, \chi_3}_{\chi_3, \chi_3}=0.
\end{array}
\end{equation}

Let us consider now a collection of sections
\begin{equation*}
\{r_{G, \psi, \chi} \in H^0(\mathcal{O}_Q(D_{G, \psi}) \otimes
L_{\chi}^{-1} ) \}_{\psi \in \{\chi_1, \chi_3 \}, \, \chi \in
S_{G, \psi}},
\end{equation*}
where
\begin{equation*}
S_{G, \chi_1}:=\{\chi_0, \, \chi_1, \chi_2 \}, \quad S_{G,
\chi_3}:=\{\chi_0, \, \chi_2, \chi_3 \}.
\end{equation*}
Let $h_{G, \psi, \chi}$ be a local representative of $r_{G, \psi,
\chi}$ on the open set $U$ and define
\begin{equation*}
\tau_{G, \psi}:= \sum_{\substack{\psi \in \{\chi_1, \chi_3 \} \\
\chi \in S_{G, \psi }}} h_{G, \psi, \chi}  w_{\chi}.
\end{equation*}
Then the natural deformation  of the $G$-cover $u \colon X\to Q$,
associated to the collection of sections $\{ r_{G, \psi, \chi} \}$,
is the subvariety $X'$ of $W$ locally defined by
\begin{equation*}
w_\chi w_{\chi'}= \bigg( \prod_{\psi \in \{\chi_1, \chi_3
\}}(\tau_{G, \psi})^{\epsilon^{G, \psi}_{\chi, \chi'}} \bigg)
w_{\chi\chi'},
\end{equation*}
 together with the map $u' \colon X'\to Q$ obtained by
restricting the projection $\pi \colon W \to Q$ to $X'$.

Coming back to our particular case, we have
\begin{equation*} D_{G, \chi_1} \in |\mathcal{O}_Q(4,
\,4)|, \quad D_{G, \chi_3}=0,
\end{equation*}
\begin{equation*}
 L_{\chi_1} \cong \mathcal{O}_Q(1,\, 1),
\quad L_{\chi_2} \cong \mathcal{O}_Q(2, \, 2), \quad L_{\chi_3}
\cong \mathcal{O}_Q(3, \, 3),
\end{equation*}
and $B=D_{G, \chi_1}$. Since $D_{G, \chi_3}=0$, the natural
deformations of $X$ are parameterized by the vector space
\begin{equation} %\label{nat.def}
 \bigoplus_{\chi\in S_{G, \chi_1} } H^0(\oo_Q(D_{G,
\chi_1})\otimes L^{-1}_\chi)
\end{equation}
\begin{equation*}
= H^0(\oo_Q(4, \, 4)) \oplus  H^0(\oo_Q(3, \, 3)) \oplus
H^0(\oo_Q(2, \, 2))\cong \mathbb{C}^{50}.
\end{equation*}

\subsection{Example where $\textrm{Sing}(X)=16\times\frac{1}{4}(1, 1)$}

Assume that, locally around each of the fixed points, the action
of $G = \langle \zeta \,| \, \zeta^4=1  \rangle$ is given by
$\zeta \cdot (x, \, y)=(\zeta x, \, \zeta y)$. In this case,
\begin{equation*}
\textrm{Sing}(X) = 16 \times \frac{1}{4}(1,1).
\end{equation*}
By using Proposition \ref{invariants-S}, we obtain
\begin{equation*}
p_g(S)=1, \quad q(S)=0, \quad K_S^2=-8,
\end{equation*}
hence $S$ is not a minimal model.

\begin{theorem} \label{teo.example.2}
The following holds:
\begin{itemize}
\item[$\boldsymbol{(i)}$] $h^2(\Theta_X)=14;$
\item[$\boldsymbol{(ii)}$] all natural deformations of $u \colon X
\to Q$ preserve the $16$
 points of type $\frac{1}{4}(1, \, 1);$
\item[$\boldsymbol{(iii)}$] there exists a $12$-dimensional family
of $\mathbb{Q}$-Gorenstein deformations of $X$, smoothing all the
singularities. The general element $X_t$ of this deformation is a
smooth, minimal surface of general type with $p_g(X_t)=1$,
$q(X_t)=0$ and $K_{X_t}^2=8;$ \item[$\boldsymbol{(iv)}$] $X_t$ is
isomorphic to a Todorov surface with $K^2=8$.
\end{itemize}
\end{theorem}
\begin{proof}
$\boldsymbol{(i)}$ By using Grothendieck duality and
K$\ddot{\textrm{u}}$nneth formula as in Proposition
\ref{prop.example.1} we obtain
\begin{equation*}
\begin{split}
H^2(\Theta_X)^\ast &= H^0(\Omega_Z^1 \otimes \Omega_Z^2)^G \\
 &=  \bigoplus_{\chi \in \widehat G}\big[ (H^0({g_1}_*\omega_{C_1})^\chi \otimes H^0({g_2}_* \omega_{C_2}^2)^{\chi^{-1}}) \\
& \quad \oplus  (H^0({g_1}_* \omega_{C_1}^2)^\chi \otimes H^0({g_2}_*\omega_{C_2})^{\chi^{-1}}) \big] \\
&=  (H^0(\oo_{\mathbb{P}^1})\otimes H^0(\oo_{\pp^1}(2))) \oplus (H^0(\oo_{\pp^1}(1))\otimes H^0(\oo_{\pp^1}(1)))\\
& \quad \oplus (H^0(\oo_{\pp^1}(1))\otimes H^0(\oo_{\pp^1}(1)))\oplus
(H^0(\oo_{\pp^1}(2))\otimes H^0(\oo_{\pp^1})),
\end{split}
\end{equation*}
which yields $h^2(\Theta_X)=14$.

$\boldsymbol{(ii)}$ The $G$-cover $u \colon X\to
 Q$ is determined by  the building data \eqref{building}, with
\begin{equation*}
D_{G, \chi_1} \in |\mathcal{O}_Q(4, \, 0)|, \quad D_{G, \chi_3}
\in |\mathcal{O}_Q(0, \,4)|,
\end{equation*}
\begin{equation*}
 L_{\chi_1} \cong \mathcal{O}_Q(1, \, 3),
\quad L_{\chi_2} \cong \mathcal{O}_Q(2, \, 2), \quad L_{\chi_3}
\cong \mathcal{O}_Q(3, \, 1).
\end{equation*}
The natural deformations of $u$ are parameterized by the vector
space
\begin{equation} %\label{nat.def-1}
\begin{split}
& \bigoplus_{\psi \in \{\chi_1, \chi_3 \}} \bigg(
\bigoplus_{\chi\in S_{G, \psi} } H^0(\oo_Q(D_{G, \psi})\otimes
L^{-1}_\chi) \bigg) \\
& = H^0(\oo_Q(4, \, 0))\oplus H^0( \oo_Q(0, \, 4)).
\end{split}
\end{equation}
%where
%\begin{equation*}
%S_{G, \chi_1}:=\{\chi_0, \, \chi_1, \chi_2 \}, \quad S_{G, \chi_3}:=\{\chi_0, \, \chi_2, \chi_3 \},
%\end{equation*}

Therefore they form a family of dimension $10$, which is exactly
the one obtained by keeping the branch divisor $B \subset Q$ of
product type. In particular, all the natural deformations preserve
the sixteen singular points of $X$.

$\boldsymbol{(iii)}$ For simplicity, set $w_i=w_{\chi_i}$ and
$\tau_{G, \chi_i}=h_i w_0$.  Writing $w_0=1$, the local equations
defining the family of natural deformations of $u \colon X \to Q$
are the following:
\begin{equation} \label{eq.nat.def.X}
\begin{array}{lll}
w_1^2=h_3w_2, & w_1w_2=h_3w_3, & w_1w_3=h_1 h_3, \\
w_2^2=h_1h_3, & w_2w_3=h_1w_1, & w_3^2=h_1w_2.
\end{array}
\end{equation}

Relations \eqref{eq.nat.def.X} can be written in determinantal
form in two different ways, namely
\begin{equation*}
\begin{split}
\boldsymbol{(a)} & \; \; {\rm rank}\left(
\begin{array}{cccc}
w_2 & w_3 & w_1 & h_1 \\
w_1 & w_2 & h_3 & w_3 \\
\end{array}\right) \le 1, \\
\boldsymbol{(b)} & \; \; {\rm rank}\left(
\begin{array}{ccc}
h_3 & w_1 & w_2  \\
w_1 & w_2 & w_3 \\
w_2 & w_3 & h_1 \\
\end{array}\right) \le 1.
\end{split}
\end{equation*}

In the sequel we will only consider the determinantal
representation $\boldsymbol{(b)}$. We can deform it by using the
parameter $s \in H^0(L_{\chi_2})=\mathbb{C}^9$, i.e.
\begin{equation} \label{mat.smoothing}
{\rm rank}\left(
\begin{array}{ccc}
h_3 & w_1 & w_2  \\
w_1 & w_2+s & w_3 \\
w_2 & w_3 & h_1 \\
\end{array}\right)
 \le 1.
\end{equation}
It is no difficult to check that for general $s \neq 0$ one obtains
a smooth surface, hence \eqref{mat.smoothing} provides a smoothing
$\pi \colon \mathcal{X} \to T$ of $X$. This is actually a
$\qq$-Gorenstein smoothing of $X$, since it is the globalization of
the local $\mathbb{Q}$-Gorenstein smoothing of the quotient
singularity $\frac{1}{4}(1, \, 1)$, see \cite[Chapter 4]{Man08}.
Therefore the general fibre $X_t$ of $\pi$ is a surface of general
type whose invariants are
\begin{equation*}
p_g(X_t)=1, \quad q(X_t)=0, \quad K_{X_t}^2=8.
\end{equation*}
The canonical divisor $K_X$ is big and nef (since
$4K_X=u^*\oo_Q(4,\, 4)$), so $K_{X_t}$ is big and nef too, as
$X_t$ is obtained by a $\mathbb{Q}$-Gorenstein smoothing of $X$.
This shows that $X_t$ is a minimal model.

In order to give a more concrete description of $X_t$, let us look
again at the double cover $v \colon Y \to E_1 \times E_2$
constructed in Section \ref{sec.example.Y}. By Proposition
\ref{prop.def(Y)} we know that $\rm{Def}(Y)$ is smooth at $[Y]$ of
dimension $18$; moreover the general deformation $Y_t$ of $Y$ is a
double cover $v_t \colon Y_t \to A_t$ of an abelian variety $A_t$,
branched on a smooth divisor $\Xi$ which is a polarization of type
$(4,\,4)$. Let us compute the dimension of the subspace of
$\rm{Def}(Y)$ consisting of surfaces for which it is possible to
lift the natural involution $\iota_t \colon A_t \to A_t$ to an
involution $\tilde{\iota}_t \colon Y_t \to Y_t$ such that
$Y_t/\tilde{\iota}_t$ is smooth. By \cite[Corollary 4.7.6]{BL04},
the divisor $\Xi$ does not contain any of the $16$ fixed points of
$\iota_t$. If we write locally the equation of
 the double cover $v_t \colon Y_t \to A_t$ as $z^2=f(x, \, y)$ so that
 $\iota_t$ is given by $(x, \, y) \to (-x, \, -y)$, we see that
$\iota_t$ lifts to $Y_t$ if an only if the branch locus $f(x, \,
y)=0$ is $\iota_t$-invariant; moreover in this case there is a
unique lifting such that the quotient is smooth; it is locally
given by $(x, \, y, \, z) \to (-x, -y, -z)$. By \cite[Corollary
4.6.6]{BL04}, the divisors in $|\Xi|$ which are invariant under
$\iota_t$ form a family of dimension
$\frac{1}{2}h^0(\oo_A(\Xi))+2-1=9$ and so, taking into account the
three moduli of abelian surfaces, we obtain a $12$-dimensional
family $\{Y_t\}$ of deformations of $Y$ which admit a lifting of
$\iota_t$.

One can further check that the lifted involution $\tilde{\iota}$
is fixed-point free and that the family $\{X_t \}$ constructed
before can be obtained as $X_t=Y_t/\tilde{\iota_t}$.

$\boldsymbol{(iv)}$ Let us consider the Kummer surface
$\textrm{Kum}(A_t):=A_t/\iota_t$. By $\boldsymbol{(iii)}$ a
general fibre $X_t$ of the $\mathbb{Q}$-Gorenstein smoothing of
$X$ is a double cover of $\textrm{Kum}(A_t)$ branched over the
$16$ nodes of $\textrm{Kum}(A_t)$ and the image $D$ of the curve
$\Xi$.

On the other hand, $\textrm{Kum}(A_t)$ can be embedded in
$\mathbb{P}^3$ as a quartic surface with $16$ nodes and via this
embedding the curve $D$ is obtained by intersecting
$\textrm{Kum}(A_t)$ with a smooth quadric surface $\Phi$ which
does not contain any of the nodes.

This shows that $X_t$ belongs precisely to the family of surfaces
with $p_g=1$, $q=0$ and $K^2=8$ constructed by Todorov in
\cite{To81}.
\end{proof}

\begin{remark} \label{rem.Todorov}
Let us fix the abelian surface $A$ and the embedding
$\emph{Kum}(A) \hookrightarrow \mathbb{P}^3$. Then the choice of
the deformation parameter $s \in H^0(L_{\chi_2})$ corresponds to
the choice of the quadric surface $\Phi \in
|\mathcal{O}_{\mathbb{P}^3}(2)|$. By \emph{\cite[Lemma 2.1]{To81}}
there is a quadric surface $\Phi_k$ in $\PP^3$ which contains
exactly $k$ $(1\le k\le 6)$ of the nodes of $\emph{Kum}(A)$ that
are general position. This means that the pullback in $A$ of the
curve $D_k:=\emph{Kum}(A) \cap \Phi_k$ is a polarization of type
$(4, \, 4)$ which contains exactly $k$ of the fixed points of
$\iota \colon A \to A$.

Therefore arguments similar to those used in the proof of Theorem
\emph{\ref{teo.example.2}}, part $\boldsymbol{(ii)}$ show that
there exists a partial $\QQ$-Gorenstein smoothing of $X$, whose
general fibre $X_t$ is isomorphic to the double cover of
$\emph{Kum}(A)$ branched over the curve $D_k$ and the remaining
$16-k$ nodes of $\emph{Kum}(A)$. The surface $X_t$ is not smooth,
since it contains exactly $k$ singular points of type
$\frac{1}{4}(1, \, 1)$. Its minimal resolution of singularities is
a Todorov surface with
 $K^2=8-k$ $(1\le k\le 6)$.
\end{remark}

%%%%%%%%%%%%%%%%%%%%%%%%%
%%%%%%%%%%%%%%%%%%%%%%%%%

\subsection{Example where  $\textrm{Sing}(X)=8\times\frac{1}{4}(1, 3)+ 8\times\frac{1}{4}(1, 1)$}

We can also twist the action of $G$ on $Z$ in such a way that
\begin{equation*}
\textrm{Sing}(X) = 8 \times \frac{1}{4}(1,1)+8 \times
\frac{1}{4}(1,3).
\end{equation*}

By using Proposition \ref{invariants-S}, we obtain
\begin{equation*}
p_g(S)=3, \quad q(S)=0, \quad K_S^2=0,
\end{equation*}
hence $S$ is not a minimal model.

Rasdeaconu and Suvaina  give an explicit construction of $S$ in
\cite[Section 3]{RS06}, showing that it is a simply connected,
minimal, elliptic surface with no multiple fibers. One can also
prove that $H^2(\Theta_X)\ne 0$, see \cite[Section 3]{LP11}.

\begin{proposition} \label{pro.example.3}
The following holds:
\begin{itemize}
\item[$\boldsymbol{(i)}$] all natural deformations of $X$ preserve the $8$
 points of type $\frac{1}{4}(1, \, 1);$
\item[$\boldsymbol{(ii)}$] there exists a family of
$\mathbb{Q}$-Gorenstein deformations of $X$, smoothing all the
singularities. The general element of this family is a smooth,
minimal surface of general type with $p_g=3$, $q=0$ and $K^2=8$.
\end{itemize}
\end{proposition}
\begin{proof}
$\boldsymbol{(i)}$ The abelian $G$-cover $u \colon X \to Q$ is
determined by the building data \eqref{building}, with
\begin{equation*}
D_{G, \chi_1}, \, D_{G, \chi_3},  \in |\oo_Q(2, \, 2)|.
\end{equation*}
\begin{equation*}
L_{\chi_1}, \, L_{\chi_2}, L_{\chi_3} \cong \oo_Q(2, \, 2).
\end{equation*}

The same argument of Theorem \ref{teo.example.2}, part
$\boldsymbol{(ii)}$ shows that the natural deformations of $X$ are
parameterized by the vector space
\begin{equation*}
\begin{split}
& H^0(\oo_Q(2, \, 2))\oplus H^0( \oo_Q(2, \, 2)) \\
 \oplus & H^0(\oo_Q) \oplus H^0(\oo_Q) \oplus H^0(\oo_Q) \oplus
H^0(\oo_Q).
\end{split}
\end{equation*}
Writing $w_i:=w_{\chi_i}$ we have
\begin{equation*}
h_1=g_1+c_1w_1 + c_2 w_2, \quad h_3=g_3+d_2w_2+d_3w_3,
\end{equation*}
where $g_i$ a local equations of $D_{G, \, \chi_i}$ and $c_i, \, d_i
\in \mathbb{C}$. Therefore the equations of the natural deformations
of $X$ are
\begin{equation} \label{eq.ex3}
\begin{split}
w_1^2& =(g_3+d_2w_2+d_3w_3)w_2, \\
w_1w_2& =(g_3+d_2w_2+d_3w_3)w_3, \\
w_1w_3&=(g_1+c_1w_1+c_2w_2)(g_3+d_2w_2+d_3w_3), \\
w_2^2& =(g_1+c_1w_1+c_2w_2)(g_3+d_2w_2+d_3w_3),\\
w_2w_3&=(g_1+c_1w_1+c_2w_2)w_1, \\
w_3^2&=(g_1+c_1w_1+c_2w_2)w_2.
\end{split}
\end{equation}
For a general choice of the parameters the morphism $\bar{u} \colon
\bar X \to Q$ is \emph{not} a Galois cover and an easy computation
shows that its branch locus is of the form
\begin{equation*}
D_{\bar{X}}=D_1+ \ldots +D_6
\end{equation*}
where the $D_i$ belong to the pencil generated by $D_{G, \, \chi_1}$
and $D_{G, \, \chi_3}$.  Then the singular locus of $D_{\bar{X}}$ is
given by the $8$ points $D_{G, \, \chi_1} \cap D_{G, \, \chi_3}$ and
$\textrm{Sing}(\bar{X})$ consists of the $8$ points of type
$\frac{1}{4}(1,\,1)$ locally defined by setting
\begin{equation*}
g_1=g_3=w_1=w_2=w_3=0
\end{equation*}
in \eqref{eq.ex3}.

$\boldsymbol{(ii)}$ We note that the set of natural deformations
$\bar X$ of $X$ which keep the $G$-action is parameterized by the
vector space $H^0(\oo_Q(2,\,2)) \oplus H^0(\oo_Q(2,\,2))$. In fact,
the action of the generator $i=\sqrt{-1}$ of $G$ must be given by
\begin{equation*}
w_1 \mapsto -iw_1, \quad w_2 \mapsto -w_2, \quad w_3 \mapsto iw_3
\end{equation*}
and substituting  in \eqref{eq.ex3} we obtain $c_1=c_2=d_1=d_3=0$.

The $G$-cover $ \bar X\to Q$ factors into two double covers
\begin{equation*}
\bar X\to K\stackrel{p}{\rightarrow} Q
\end{equation*}
where $K$ is a $K3$ surface with $8$ ordinary double points and $p
\colon  K\to Q$ is a double cover branched over $D_{G, \chi_1}+D_{G,
\chi_3}$. Let $D_{G, \chi_2}$ be a general member in the pencil
induced by $D_{G, \chi_1}$ and $D_{G, \chi_3}$. Let $\bar D_{G,
\chi_2}=p^*D_{G, \chi_2}$ and $2\bar D_{G, \chi_i}=p^*D_{G, \chi_i}$
for $i=1, 3$. Since $D_{G, \chi_2}$ is linearly equivalent to $D_{G,
\chi_i}$ for $i=1, 3$ and a $K3$ surface is simply connected, $\bar
D_{G, \chi_2}$ is linearly equivalent to $\bar D_{G, \chi_1}+\bar
D_{G, \chi_3}$. Note that both these curves have exactly 8 nodes.
The double cover $\tilde X$ of $K$ branched over $\bar D_{G,
\chi_2}$ is deformation equivalent to $\bar X$, and $\tilde X$ can
be realized as the bidouble cover of $Q$ branched over $D_{G,
\chi_1}$, $D_{G, \chi_3}$ and $D_{G, \chi_2}$. Therefore if one
deforms $D_{G, \chi_2}$ to a general divisor of bidegree $(2, 2)$ we
have a $\QQ$-Gorenstein smoothing of $\tilde X$ which smoothes all
the singularities. Since $\bar X$ is a deformation of $X$ and
$\tilde X$ is deformation equivalent to $\bar X$, we have a smooth
projective surface in the deformation space of $X$ which is a
$\QQ$-Gorenstein smoothing of $\tilde X$. Finally, we note that each
deformation is a $\QQ$-Gorenstein one. In fact, $\tilde X$ and $\bar
X$ are double covers of the $K3$ surface $K$ branched over $\bar
D_{G, \chi_2}$ and $\bar D_{G, \chi_1}+\bar D_{G, \chi_3}$,
respectively. Let ${\mathcal X} \to \Delta$ be a family of double
covers of $K$ obtained deforming the branch locus from $\bar D_{G,
\chi_1}+\bar D_{G, \chi_3}$ to $\bar D_{G, \chi_2}$. By using the
canonical divisor formula for a double cover, it is not hard to see
that $K_{\mathcal X}$ is a $\QQ$-Cartier divisor. Therefore the
transitive property of $\QQ$-Gorenstein deformations implies that
$X$ has a $\QQ$-Gorenstein smoothing.
\end{proof}

\begin{remark} \label{rem.partial smoothing}
By applying arguments similar to those used in Remark
\emph{\ref{rem.Todorov}} and in \emph{\cite[Section 2]{Lee10}}, one
can construct surfaces of general type with $p_g=3$, $q=0$ and
$K^2=k$ $(2\le k\le 8)$ by first taking a $\QQ$-Gorenstein smoothing
of $k$ singular points of type $\frac{1}{4}(1, \, 1)$ of $\bar X$
and then the minimal resolution of the remaining $8-k$ singular
points of the same type.
\end{remark}

%%%%%%%%%%%%%%%%%%%%%%%%%%%%%%%%%
% References
%%%%%%%%%%%%%%%%%%%%%%%%%%%%%%%%%

\end{document}